\documentclass[12pt, reqno]{amsart}
\usepackage[utf8]{inputenc}
\usepackage[english]{babel}
\usepackage{standalone}

\usepackage{ifthen}
\usepackage{microtype}
\usepackage[dvipsnames]{xcolor}
\usepackage{comment}
\usepackage{blkarray}
\usepackage{fullpage}
\usepackage{mathtools}
\usepackage{xfrac}
\usepackage{shuffle}
\usepackage{centernot}
\usepackage{enumitem}
\usepackage{mathdots}
\usepackage{multirow}
\usepackage{amssymb}
\usepackage{soul}
\usepackage{booktabs}
\usepackage{diagbox}
\usepackage{multicol}
\usepackage[noadjust]{cite}

\usepackage{amsmath}

\usepackage{algorithm}
\usepackage[noend]{algpseudocode}

\newlist{paraenum}{enumerate}{1}
\setlist[paraenum]{wide, label=(\arabic*)}

\makeatletter
\patchcmd{\@setaddresses}{\indent}{\noindent}{}{}
\patchcmd{\@setaddresses}{\indent}{\noindent}{}{}
\patchcmd{\@setaddresses}{\indent}{\noindent}{}{}
\patchcmd{\@setaddresses}{\indent}{\noindent}{}{}
\makeatother

\usepackage{caption}
\usepackage{subcaption}
\usepackage{tikz}    
\usetikzlibrary{cd}
\usetikzlibrary{backgrounds}
\usetikzlibrary{shapes}
\usetikzlibrary{arrows}
\usetikzlibrary{positioning}
\usetikzlibrary{calc}
\usetikzlibrary{quotes}
\usetikzlibrary{fit}
\usetikzlibrary{decorations.pathreplacing}

\usepackage[pdftex]{hyperref}
\hypersetup{
    colorlinks=true,
    linkcolor={magenta},
    citecolor={blue},
    urlcolor={blue}
}
\usepackage[capitalise]{cleveref}

\numberwithin{equation}{section}
\theoremstyle{plain}
\newtheorem{theorem}{Theorem}[section]
\newtheorem{conjecture}[theorem]{Conjecture}

\newtheorem{lemma}[theorem]{Lemma}
\newtheorem{proposition}[theorem]{Proposition}

\theoremstyle{definition}
\newtheorem{remark}[theorem]{Remark}

\newtheorem{examplex}[theorem]{Example}
\newenvironment{example}
  {\pushQED{\qed}\examplex}
  {\popQED\endexamplex}

\newtheorem{definition}[theorem]{Definition}

\newcommand{\rr}{\mathbb{R}}

\newcommand{\pp}{\mathbb{P}}

\newcommand{\co}{\mathcal{O}}

\newcommand{\cm}{\mathcal{M}}

\DeclareMathOperator{\image}{image}

\newcommand{\pt}{\mathrm{PT}}

\newcommand{\gr}{\mathrm{Gr}}

\DeclareMathOperator{\Proj}{Proj}

\DeclareMathOperator{\bP}{\mathbb{P}}
\DeclareMathOperator{\bC}{\mathbb{C}}
\DeclareMathOperator{\bR}{\mathbb{R}}
\DeclareMathOperator{\cO}{\mathcal{O}}
\DeclareMathOperator{\Gr}{\mathrm{Gr}}
\DeclareMathOperator{\cA}{\mathcal{A}}

\newcommand{\reg}{\operatorname{reg}}

\definecolor{benpurple}{RGB}{180, 0, 240}

\title{Log Canonical Models and Positive Geometries}

\author{Benjamin Hollering}
\address{Max Planck Institute for Mathematics in the Sciences, Inselstra{\ss}e 22, 04103 Leipzig, Germany}
\email{benjamin.hollering@mis.mpg.de}

\author{Dmitrii Pavlov}
\address{Max Planck Instiute for Physics, Boltzmannstra{\ss}e 8, 85748 Garching bei M\"unchen, Germany}
\email{pavlov@mpp.mpg.de}

\author{Elizabeth Pratt}
\address{Perimeter Institute for Theoretical Physics,  Waterloo, Ontario, Canada}
\email{lpratt@perimeterinstitute.ca}

\subjclass{Primary: 14E25. Secondary: 14Q15, 05E14, 14E30.}
\begin{document}

\maketitle
\begin{abstract}
    Constructing log canonical compactifications of open varieties is a central problem in birational geometry. Finding a natural coordinate system and obtaining the equations of these models is difficult in general. We show that for a large class of varieties explicit coordinates for the log canonical model are provided by canonical forms of positive geometries, and use this to compute the equations of these models. Our theory applies, for instance, to complements of hyperplane arrangements, cubic surfaces with lines removed, and the moduli space of marked cubic del Pezzo surfaces.
\end{abstract}

\section{Introduction}
\label{sec:intro}
A central problem in birational geometry is to construct compactifications of open varieties such that the boundary of the compactification has prescribed properties (e.g., only has mild singularities). There are several strategies for doing this, including tropical compactifications \cite{tevelev2007tori} and Chow quotients \cite{kapranov1993chow}. Given a compactification, one may obtain another compactification by blowing up singular points or contracting divisors, so that there are many possible compactifications of a variety depending on the types of boundary singularities one allows. In this article, we concentrate on the \emph{log canonical compactification}. It is intrinsic to the open variety and plays a central role in the minimal model program. We present a new strategy to obtain explicit embeddings of such compactifications, using tools from positive geometry \cite{whatisPG}, a novel field of mathematics inspired by theoretical physics.

We now define these compactifications. The necessary technical background, including the definition of log canonical singularities, is provided in Section \ref{sec:prelim-birational}.
 For a closed subvariety $Z$ of a variety $X,$ we use $Z_{\rm{red}}$ to denote the closed subscheme on $Z$ with the reduced structure.

\begin{definition}\label{def:logcanonical}
Let $U$ be a smooth complex algebraic variety. A \emph{log canonical compactification} of $U$ is a proper variety $X$ and an open immersion
\[
j : U \hookrightarrow X
\]
such that
\(
D := (X \setminus U)_{\mathrm{red}}
\) is pure of codimension one,
the pair $(X,D)$ has log canonical singularities and $K_X+D$ is an ample Cartier divisor.
\end{definition}

If such a compactification exists, it is unique (Proposition \ref{prop:lc unique}). However, it is in general difficult to prove that it exists. It is also difficult to find natural coordinates and explicit equations for log canonical compactifications. For complements of hyperplane arrangements, this problem was treated in \cite{tevelev2006hyperplanes}. The specific case of the moduli space $\mathcal{M}_{0,n}$ is the topic of \cite{KeelTevelevM0n,Monin2017,Gillespie2022}. For the moduli space of marked cubic del Pezzo surfaces, the log canonical compactification was described in \cite{HKTdelPezzo} but, to the best of our knowledge, explicit equations are not known. For a description of the log canonical compactifications of configuration spaces of six, seven and eight points in the projective plane $\mathbb{P}^2$, see \cite{luxton2008thesis}, \cite{corey2021initial} and \cite{corey2023x38}, respectively.

In this article, we present a new approach to obtaining natural coordinates and defining equations of log canonical compactifications of a large class of varieties. 
This approach rests on the framework of \emph{positive geometries} \cite{Arkani_Hamed_2017} introduced by Arkani-Hamed--Bai--Lam.
Positive geometries are mathematical objects that play a significant role in theoretical particle physics. 
A positive geometry is a pair of a complex algebraic variety $X$ defined over $\rr$ and a semialgebraic set $X_{\geq 0}\subset X_{\mathbb{R}}$ inside the real locus of $X$, equipped with a rational differential form $\omega_{(X,X_{\geq0})}$ on $X$, called the \emph{canonical form}, satisfying a number of technical properties.
We present a complete definition in Section \ref{sec:prelim-posgeom}. 

A key feature of positive geometries is that the construction of canonical forms is often algorithmic once the boundary of $X_{\geq 0}$ is understood. When the semialgebraic set $X_{\geq 0}$ is the closure of a connected component of an open variety $(X\setminus Y)_{\mathbb{R}}$, its canonical form is an element of $H^0(X,K_X +Y)$. In favorable situations, the canonical forms associated with such connected components span $H^0(X,K_X +Y)$ and, moreover, parametrize the log canonical compactification. The main result of this article, Theorem \ref{thm:mainnew}, describes the conditions under which this is the case. Our approach relies on the novel framework of \cite{brown-dupont} connecting positive geometries and mixed Hodge theory.
We note that the class of varieties satisfying the conditions of Theorem \ref{thm:mainnew} includes all of the varieties mentioned above, and we compute the equations of their log canonical compactifications. The code used to perform these computations is available at \cite{zenodo}.

The remainder of this article is organized as follows. In Section \ref{sec:prelim} we provide the necessary definitions from birational geometry and review the machinery of positive geometries. In Section \ref{sec:main}, we prove Theorem \ref{thm:mainnew}, which describes when the log canonical compactification of an open variety can be embedded by canonical forms. This result unifies and generalizes several instances of this phenomenon observed in prior work such as \cite{parketaylorvar, cubicsurfaces}. In the remaining sections, we describe how these previous results can be seen as instances of Theorem \ref{thm:mainnew} as well as introduce some new families whose log canonical compactifications can also be obtained from our framework. 

In Section \ref{sec:hyperplanes}, we review the construction of log canonical compactifications of hyperplane arrangement complements due to Hacking--Keel--Tevelev \cite{tevelev2006hyperplanes} and show that the compactification given by canonical forms yields the same variety. We also show it is defined by quadrics scheme-theoretically. Section \ref{sec:nonlinear} is devoted to nonlinear examples. We compute the log canonical compactification of a cubic surface with its $27$ lines removed. For the configuration space $\mathcal{X}(3,6)$ of $6$ points in general position in $\mathbb{P}^2$ and the closely related moduli space of marked cubic del Pezzo surfaces $\mathcal{Y}(3,6)$ we describe parametrizations of the log canonical compactifications and compute all quadratic polynomials vanishing on these compactifications. Numerical evidence suggests that the corresponding ideals are quadratically generated. We also briefly discuss the implications of our approach for the spaces $\mathcal{X}(3,7)$ and $\mathcal{Y}(3,7)$, and for $\cm_{0,n}$ via Parke--Taylor varieties \cite{parketaylorvar}.

\section{Preliminaries}
\label{sec:prelim}

\subsection{Preliminaries on Birational Geometry}\label{sec:prelim-birational}
In this section we give the background on birational geometry that is necessary to explain Definition \ref{def:logcanonical}. We prove the uniqueness of the log canonical compactification in Proposition \ref{prop:lc unique}, as we could not find it in the literature in the form we required. 

We first fix some conventions. All our varieties will be complex and normal. In particular, a Cartier divisor determines a unique Weil divisor. When we speak of the (Weil) divisor associated with a codimension one subvariety, we take each component with multiplicity one.

For a proper birational map $f: Y \to X$ of varieties and a subvariety $Z \subset X$ we use $f_*^{-1}(Z)$ to denote the proper transform of $Z$. That is, let  $f^{-1}$ be the birational inverse of $f$ and $Z_0$ be the intersection of $Z$ with the domain of definition of $f^{-1}.$ Then $f_*^{-1}(Z)$ is defined as the Zariski closure of $f^{-1}(Z_0)$ in $Y.$ For a sum $D = \sum_i a_i D_i$ of prime divisors we define $(f_*)^{-1}D := \sum_i a_i (f_*)^{-1} D_i.$

It is convenient to allow compactifications that may be singular. The class of allowed singularities in a log canonical compactification is called \emph{log canonical}. For more background on birational geometry and different types of singularities, see \cite[Chapter Two]{Kollar}.

\begin{definition}[Log canonical singularities]\label{def:lc singularities}

Let $(X,D)$ be a pair where $X$ is a normal variety, $D$ is a sum of distinct prime divisors, and $K_X + D$ is $\mathbb{Q}$-Cartier. 

Suppose $Y$ is another normal variety  and $f: Y \to X$ is birational with exceptional locus $\rm{Ex}(f)$ Then there exists $m$ such that $m(K_X + D)$ is Cartier, and $ma(E,f)$ are integers defined by the pullback formula
\[\cO_Y(m(K_Y + f_*^{-1}D)) \cong f^*\cO_X(m(K_X + D))(\sum_{E \in \rm{Ex}(f)} ma(E,f) \cdot E),\]
where $E$ runs over the irreducible components of $\mathrm{Ex}(f)$.
We define the \emph{discrepancy} of the pair $(X,D)$ to be the infimum of $a(E,f)$ as $E$ runs over all exceptional divisors of all such maps. We say $(X,D)$ has \emph{log canonical singularities} if it has discrepancy at least $-1.$
\end{definition}

It is common in the birational geometry literature to begin not with an open variety $U,$ but rather with a pair $(X,D),$ and ask for a log canonical model of the pair. As such, we could not find a result on the uniqueness of the log canonical compactification in the literature in the form we required, so we provide a proof ourselves. It closely follows the proof of the uniqueness of the log canonical model of a pair $(X,D)$ given in \cite[Theorem 3.52]{Kollar}. 

\begin{proposition}\label{prop:lc unique}
    Let $U$ be a smooth complex algebraic variety. Suppose that $j_1: U \hookrightarrow X_1$ and $j_2: U \hookrightarrow X_2$ are two log canonical compactifications of $U$ with reduced boundary divisors $D_1 = (X_1 \setminus U)_{\rm{red}}$ and $D_2 = (X_2 \setminus U)_{\rm{red}}.$ Then $X_1 \cong X_2$ and $D_1 \cong D_2.$
\end{proposition}
\begin{proof}
    Let $\phi: X_1 \dashrightarrow X_2$ be the rational map sending $j_1(U)$ to $j_2(U).$ We begin by using Hironaka's resolution of singularities \cite{hironaka} to find a common resolution 
    \[
\begin{tikzcd}
& Y \arrow[dl, "p_1"'] \arrow[dr, "p_2"] & \\
X_1 \arrow[rr, dashed, "\phi"] && X_2.
\end{tikzcd}
\]
Indeed, by Hironaka's theorem, there exists a smooth variety $Y$ and maps $p_1, p_2$ which are isomorphisms over $U$ such that $\Delta := Y \setminus U$ is a closed subset of codimension one; see also \cite[Theorem 0.2]{Kollar} for the precise formulation of Hironaka's theorem used. Then as subsets $\Delta = \mathrm{Ex}(p_i) \cup (p_i)_*^{-1}(D_i)$ and as Weil divisors $\Delta = \mathrm{Ex}(p_i)_{\mathrm{red}} + (p_i)_*^{-1}(D_i)$ for $i = 1, 2,$ discarding components of codimension $\geq 2$, which
do not affect the divisor class.

Now, $K_{X_i} + D_i$ is Cartier, and the variety $Y$ is smooth, so $K_Y + \Delta$ is also Cartier. As in the discrepancy formula of Definition \ref{def:lc singularities}, we may write
\begin{equation}\label{eq:discrepancy}
    \begin{split}
        K_Y + (p_1)_*^{-1}D_1
& =
p_1^*(K_{X_1} + D_1) + \sum_i a(E, p_1)E_i, \\
K_Y + (p_2)_*^{-1}D_2
& =
p_2^*(K_{X_2} + D_2) + \sum_i a(F, p_2)F_i,
    \end{split}
\end{equation}
where the sum is over exceptional divisors of $p_1$ and $p_2$. Now, since both $(X_1,D_1)$ and $(X_2,D_2)$ have log canonical singularities, we can add $\mathrm{Ex}(p_i)_{\rm{red}}$ to each side of \eqref{eq:discrepancy} to obtain 
\begin{equation}\label{eq:discrepancy2}
    \begin{split}
        K_Y + \Delta
& =
p_1^*(K_{X_1} + D_1) + G_1 \\
& =
p_2^*(K_{X_2} + D_2) + G_2,
    \end{split}
\end{equation}
where $G_1$ and $G_2$ are exceptional and effective. But then $(p_i)_*\cO_Y(G_i)=\cO_{X_i}.$ Thus by the projection formula, we have an isomorphism \[
H^0\!\left(Y,\mathcal{O}_Y\bigl(m(K_Y+\Delta)\bigr)\right)
\cong
H^0\!\left(X_i,\mathcal{O}_{X_i}\bigl(m(K_{X_i}+D_i)\bigr)\right),
\qquad i=1,2.
\]
In particular, there is an isomorphism
 \[\bigoplus_{m \geq 0} H^0(X_1, \cO_{X_1} (m(K_{X_1} + D_1))) \cong \bigoplus_{m \geq 0} H^0(X_2, \cO_{X_2}(m(K_{X_2} + D_2))).\]
But since $K_{X_i} + D_i$ is ample, its ring of sections recovers the variety $X_i,$ and thus $X_1 \cong X_2.$ The resulting isomorphism is the one induced by $\phi$, hence restricts to the identity on $U$ and therefore carries $D_1$ to $D_2$.
\end{proof}

The ring constructed in the proof of Proposition \ref{prop:lc unique} is called the \emph{log canonical ring}. It can be obtained from any compactification with at most log canonical singularities.
If the log canonical compactification exists, then it equals Proj of this ring \cite[Theorem 3.52]{Kollar}.

\begin{definition}[Log canonical ring]
    Let $U$ be a smooth complex algebraic variety. Let $(X,D)$ be a  compactification of $U$ with log canonical singularities. The \emph{log canonical ring} is 
    \[R(U) := \bigoplus_{m \geq 0} H^0(X, \cO_X(m(K_X + D))).\]
    If $R(U)$ is finitely generated, then $X_{lc} := \Proj R(U)$ is a projective algebraic variety over $\bC.$
\end{definition}

An invertible sheaf $\mathcal{L}$ on a normal projective variety $X$ is called \emph{big} if, for some $r>0$, the complete linear system $|\mathcal{L}^{\otimes r}|$ defines a birational map from $X$ onto its image. 
    
Suppose that for some $m>0$ the divisor $m(K_X +D)$ is Cartier and $\cO_X(m(K_X +D))$ is big. Then we obtain a birational map 
    \[\phi: X \dashrightarrow X_{lc},\]
    which is regular on $U$. If the open variety $U$ has a log canonical compactification, it equals $(X_{lc}, \phi_* D)$ \cite[Theorem 3.52]{Kollar}. Thus the log canonical ring gives an algorithmic way to access this compactification; one constructs the ring and checks the conditions. 
    
    However, checking the conditions (particularly the finite generation) can be difficult in general. When the variety $U = X$ itself is already smooth and compact, one recovers the \emph{canonical model} of $X$. These exist whenever $X$ is smooth and projective of general type \cite[Corollary 1.1.1]{Birkar2009-pi}. 

When the divisor $K_{X_{lc}} + \phi_* D$ is very ample, the resulting embedding of $X_{lc}$ is called the \emph{log canonical embedding}. In Section \ref{sec:main}, we will show that for pairs $(X,D)$ which are sufficiently nice, we can produce a basis of the degree one part of the log canonical ring via positive geometry, thus yielding this embedding.

\subsection{Preliminaries on Positive Geometries} \label{sec:prelim-posgeom}

Positive geometry is a new field of mathematics inspired by computations in quantum field theory, and by the study of positive parts of algebraic varieties. The ideas of positive geometry have developed significantly over the last decade. Recently, Brown and Dupont \cite{brown-dupont} establish a new mathematical foundation based on Deligne's mixed Hodge theory. We adopt this viewpoint first, as it is closest to birational geometry, and then recall the original framework of \cite{Arkani_Hamed_2017} which is better suited for computation.

The central object in the Brown--Dupont framework is a pair $(X, Y)$ of an $n$-dimensional complex variety $X$ and a closed (usually reducible) subvariety $Y$. The relative homology group $H_n(X,Y)$ carries a mixed Hodge structure by work of Deligne \cite{Deligne2,Deligne3}. Such a pair possesses an important statistic called the \emph{genus}, which generalizes the genus of an algebraic variety. In what follows, we assume the complement $X\setminus Y$ is smooth.

\begin{definition}[{\cite[Def. 3.1]{brown-dupont}}]
    The \emph{genus} of the pair $(X,Y)$ is the sum of the outer Hodge numbers of the relative homology group $H_n(X,Y)$ (with rational coefficients): 
 $$g(X,Y) := \sum_{p>0} h^{-p,0}(H_n(X,Y)).$$
\end{definition}

Let $U$ be the complement $X\setminus Y$. The main theorem of \cite{brown-dupont} states that for every pair $(X,Y)$ of genus zero there is a map 
\begin{equation}\label{eq:canonical map}
    \omega:\ H_n(X,Y)\to \Omega_{\log}^n(U),
\end{equation}
which associates with each homology cycle $\sigma$ a logarithmic form $\omega(\sigma)$ on $U$. This is called the \emph{canonical form} of the cycle. 
Here, $\Omega_{\log}^n(U)$ is the space of logarithmic $n$-forms on $U$, defined below. This is a finite-dimensional vector space which depends only on $U$. It is constructed by taking a smooth simple normal crossing compactification of $U,$ which exist for complex varieties by Hironaka's resolution of singularities \cite{hironaka}. Thus we may suppose that $(X, Y)$ is such  a compactification.

\begin{definition}\label{def:logforms}
Let $X$ be a smooth compact complex variety of dimension $n$, and $Y \subset X$ be a
hypersurface which is a simple normal crossing divisor. Restriction from $X$ to $U:=X\setminus Y$ gives a map \begin{equation}\label{eq:res}
    H^0(X, \Omega_X^n(Y)) \to H^0(U, \Omega^n_U), \quad s \mapsto s|_U.
\end{equation}
We denote its image $\Omega^n_{\log}(U),$ and call it the \emph{space of logarithmic $n$-forms on $U.$}
\end{definition}

The image of the restriction map is independent of the choice of a smooth simple normal crossing compactification \cite[Proposition 1.13]{brown-dupont}, so the space $\Omega_{\log}^n(U)$ is well-defined. That is, if a form is logarithmic on one SNC compactification, then it is logarithmic on every SNC compactification. For a general compactification, the image of the restriction map \eqref{eq:res} contains the subspace of the space of logarithmic forms, but may be larger \cite[Proposition 1.13]{brown-dupont}. The following is an example of a non-SNC compactification in which restrictions of some forms fail to be logarithmic.

\begin{example}\label{eg:lines prelim}
Let $\cA$ be the hyperplane arrangement in $\mathbb{P}^2$ whose real points are depicted in Figure \ref{fig:blowuppq}, where the hyperplanes are the vanishing loci of the forms $x, y, z, x-z, y-z$. Let $U := \bP^2 \setminus \cA.$ Then $\bP^2$ itself is a compactification of $U,$ but the boundary divisor $\cA$ is not simple normal crossing. We denote the two points where three lines meet by $q = [1:0:0]$ and $p = [0:1:0].$  

Let $X = \mathrm{Bl}_{p,q}\bP^2$ denote the blowup at points $p$ and $q$ and let $E_p, E_q$ denote the components of the exceptional divisor of the blowup. Also, let $Y = \tilde{L_1} \cup \ldots \cup \tilde{L_5} \cup E_p \cup E_q$, where $\tilde{L_i}$ is the proper transform of the line $L_i$. Then $(X,Y)$ is an SNC compactification of $U.$
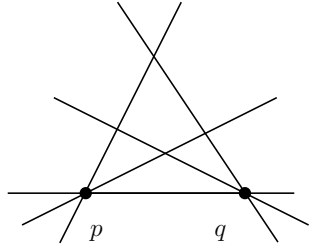
\begin{figure}[!h]
    \centering
    \scalebox{0.8}{\tikzset{every picture/.style={line width=0.75pt}} 

\begin{tikzpicture}[x=0.75pt,y=0.75pt,yscale=-1,xscale=1]

\draw    (437,170) -- (257,170) ;
\draw    (326,50) -- (426.75,200.75) ;
\draw    (366,50) -- (289.75,201.25) ;
\draw    (266,190) -- (426,110) ;
\draw    (406,170) -- (306,170) ;
\draw [shift={(306,170)}, rotate = 180] [color={rgb, 255:red, 0; green, 0; blue, 0 }  ][fill={rgb, 255:red, 0; green, 0; blue, 0 }  ][line width=0.75]      (0, 0) circle [x radius= 3.35, y radius= 3.35]   ;
\draw [shift={(406,170)}, rotate = 180] [color={rgb, 255:red, 0; green, 0; blue, 0 }  ][fill={rgb, 255:red, 0; green, 0; blue, 0 }  ][line width=0.75]      (0, 0) circle [x radius= 3.35, y radius= 3.35]   ;
\draw    (446,190) -- (286,110) ;

\draw (327,195) node   [align=left] {\begin{minipage}[lt]{27.2pt}\setlength\topsep{0pt}
$\displaystyle p$
\end{minipage}};
\draw (396,195) node   [align=left] {\begin{minipage}[lt]{13.6pt}\setlength\topsep{0pt}
$\displaystyle q$
\end{minipage}};

\end{tikzpicture}}
    \caption{A non-simple arrangement of five lines in $\bP^2$}\label{fig:blowuppq}
\end{figure}

A basis for the space $H^0(\bP^2, \Omega^2_{\bP^2}(5H))$ is 
    \[\frac{x^2 \Omega}{F}, \,\frac{xy\Omega}{F},\, \frac{y^2 \Omega}{F}, \,\frac{xz\Omega}{F},\,\frac{yz \Omega}{F}, \,\frac{z^2\Omega}{F},\]
    where $F = xyz(x-z)(y-z)$ and $\Omega = x dy \wedge dz - y dx \wedge dz + z dx \wedge dy$. However, the space $\Omega^n_{\log}(U)$ is a four-dimensional subspace of the space spanned by (restrictions of) these forms. To see this, we note the bundle $\Omega^n_X(Y)$ corresponds to the divisor $2H - E_p-E_q.$ Its sections are the pullbacks of the forms which vanish on both $p$ and $q.$ Thus the forms $\frac{x^2 \Omega}{F}$ and $\frac{y^2 \Omega}{F}$ do not extend to logarithmic forms on the simple normal crossing compactification $X.$ 
\end{example}

Another important notion in this framework is that of \emph{combinatorial rank}. The combinatorial rank $\mathrm{cr}(X,Y)$ is defined as the Hodge number $h^{0,0}(H_n(X,Y))$. When the pair $(X,Y)$ has genus zero, we have $\mathrm{cr}(X,Y) = \dim \Omega_{\log}^n(X\setminus Y)$ \cite[Sec. 4]{brown-dupont}. That is, the combinatorial rank recovers the dimension of the space of logarithmic forms. As the name suggests, this dimension can be computed purely combinatorially, as the dimension of the $n$th homology group of the Delta-complex associated with the divisor $Y$ \cite[Proposition 4.4]{brown-dupont}.

So far, this theory is quite abstract. While the formalism of \cite{brown-dupont} guarantees the existence of canonical forms of complex homology cycles of $H_n(X,Y)$, it provides no explicit recipe to compute them. On the other hand, there is a natural map $H_n(X_{\mathbb{R}}, Y_{\mathbb{R}})\to H_n(X,Y)$. For the real homology cycles that are positive geometries in the sense of \cite{Arkani_Hamed_2017} one can use the original combinatorial approach to construct canonical forms explicitly. We now review the definition of positive geometries from \cite{Arkani_Hamed_2017}. In Section \ref{sec:hyperplanes} we describe the canonical forms of polytopes, which are the simplest examples of positive geometries in projective space.

Let $X$ be an $n$-dimensional irreducible normal complex projective variety defined over $\mathbb{R}$. 
Let $X_{\geq 0}$ be an $n$-dimensional semi-algebraic subset of $X_\mathbb{R}$, the real locus of $X$. 
Suppose the interior $X_{>0}$ of $X_{\geq 0}$ (in the analytic topology on $X_\mathbb{R}$) 
is an oriented manifold, and that its closure recovers $X_{\geq 0}$. 
The \emph{algebraic boundary} $\partial_a X_{\geq 0}$ is the Zariski closure in $X$ of $X_{\geq 0} \setminus X_{>0}$. 
The irreducible components of $\partial_a X_{\geq 0}$ are prime divisors $D_1,\ldots,D_k$ on~$X$. We define $(D_i)_{\geq 0}$ as the closure in $(D_i)_\mathbb{R}$ of the interior of $D_i \cap (X_{\geq 0}\setminus X_{>0})$.

\begin{definition}[Positive geometries as in {\cite[Section 2.1]{Arkani_Hamed_2017}}]\label{def:pos_geom_lam}
The pair $(X, X_{\geq 0})$ is called a \emph{positive geometry} if there exists a unique rational $n$-form $\omega_{(X,X_{\geq 0})}$ on $X$, called the \emph{canonical form}, satisfying the following properties.
\begin{enumerate}
    \item If $n > 0$, the form $\omega_{(X,X_{\geq 0})}$ has poles only along $D_1,\ldots,D_k$. 
    \item For each irreducible component $D_i$ of $\partial_a X_{\geq 0}$, the pair $(D_i,(D_i)_{\geq 0})$ is a positive geometry; $\omega_{(X,X_{\geq 0})}$ has a simple pole along $D_i$ 
    and the Poincaré residue $\operatorname{Res}_{D_i} \omega_{(X,X_{\geq 0})}$ equals the canonical form of 
    $(D_i,(D_i)_{\geq 0})$.
    \item If $n=0$, then $X=X_{\geq 0}$ is a point and $\omega_{(X,X_{\geq 0})} = \pm 1$. 
\end{enumerate}
\end{definition}

\begin{remark} \label{rem:smooth}
    In this paper we will additionally assume $X\setminus \partial_a X_{\geq 0}$ is smooth. In particular, for each boundary component $D_i$ of the positive geometry $(X, X_{\geq 0})$, we have that $D_i\setminus \partial_a(D_i)_{\geq 0}$ is smooth. This will be the case in all the examples we consider.
\end{remark}

The paper \cite{cubicsurfaces} blends the approaches of \cite{Arkani_Hamed_2017} and \cite{brown-dupont} by imposing constraints on the real points of the pair $(X,Y)$ considered by Brown and Dupont, introducing \emph{positive arrangements}. In the present paper all of our examples will be instances of these. Positive arrangements globalize the theory of Arkani-Hamed--Bai--Lam by considering all regions of the complement $X_{\bR} \setminus Y_{\bR},$ rather than a specified non-negative locus $X_{\geq 0}.$ We review \cite[Definition 5.2]{cubicsurfaces}.

\begin{definition}[Positive arrangements]
A positive arrangement is a pair ($X,Y)$ of a complex projective variety $X$ and a hypersurface $Y$ such that
\begin{enumerate}\label{def:positive arrangement}
    \item The variety $X$ has a real smooth point,
    \item Each irreducible component of $Y$ has a real smooth point,
    \item The singular locus $\mathrm{Sing}(X)$ is contained in $Y$, and
    \item The pair $(X,Y)$ has genus zero.
\end{enumerate}
\end{definition}

\begin{example}
    Let $\cA$ be a finite union of real hyperplanes in $\bP^n.$ Then the pair $(\bP^n, \cA)$ is a positive arrangement. It clearly satisfies the first three conditions of Definition \ref{def:positive arrangement}, and the pair has genus zero by \cite[Remark 2.8]{brown-dupont}.
\end{example}

\begin{definition}
    Let $(X,Y)$ be a positive arrangement. A \emph{real region} $\sigma$ of $(X, Y )$ is the closure
of an orientable connected component of the complement $X_{\bR} \setminus Y_{\bR}.$
\end{definition} 
Each region yields an element of the relative homology group $H_n(X,Y).$ Thus the map $\omega$ associates with each region a logarithmic form.  A key feature of this setup is that, even if $(X,Y)$ is not an SNC compactification, the forms obtained as canonical forms of regions are still logarithmic on $X \setminus Y.$  This is formalized in Lemma \ref{lem:canformslogarithmic} and essentially follows from the invariance of the canonical form map under modifications of the pair $(X, Y)$; see \cite[Sec. 2.3.1]{brown-dupont}. We will exploit this fact frequently for computational purposes.

We now explain the connection between the two approaches to positive geometries. A real region $\sigma$ gives a cycle $[\sigma] \in H_n(X_{\bR}, Y_{\bR})$. If the region $\sigma$ gives rise to a positive geometry $(X,\sigma)$ in the sense of \cite{Arkani_Hamed_2017},  it has a canonical form which is unique up to multiplication by a nonzero scalar. The following lemma shows this form agrees with $\omega([\sigma])$ where $\omega$ is the map from Equation \eqref{eq:canonical map}. Thus in later sections, we use $\sigma$ to denote both the region and its corresponding cycle.
In what follows, $X^{\mathrm{reg}}$ denotes the regular locus of a variety $X$. 

\begin{lemma}\label{lem:canformslogarithmic}
    Let $X$ be a normal complex projective variety and $Y$ be a hypersurface such that $X\setminus Y$ is smooth. Let $\sigma$ be a real region of $X \setminus Y$ such that $(X,\sigma)$ is a positive geometry satisfying the conditions of Remark \ref{rem:smooth}. Then the canonical form $\omega_{(X,\sigma)}$ of $(X,\sigma)$ is an element of $\Omega^n_{\log}(U).$ In particular, $\omega_{(X,\sigma)}$ equals $\omega([\sigma])$ up to nonzero scalar. 
\end{lemma}
\begin{proof}
    We proceed by induction on the dimension of $X.$ In the dimension one case, $X$ is normal, hence smooth, of genus zero (since varieties supporting positive geometries cannot have global holomorphic top-forms). Thus $X$ is $\bP^1,$ and all (reduced) point arrangements are simple normal crossing. For the inductive step, by \cite[Proposition 1.15]{brown-dupont} it suffices to check that the residue of $\omega_{(X,\sigma)}$ along $Y^{\mathrm{reg}}$ lies in $\Omega^{n-1}_{\log}(Y^{\reg}).$ If $Y_1, \, \ldots, \, Y_k$ are the irreducible components of $Y$, then we have
    \[H^0(Y^{\reg}, \Omega^{n-1}_{Y^{\reg}}) = \bigoplus_{i = 1}^k H^0(Y_i \cap Y^{\reg}, \Omega^{n-1}_{Y_i \cap Y^{\reg}}).\]
    By the definition of $\omega_{(X,\sigma)},$ its residue along $Y_i \cap Y^{\reg}$ is zero unless $Y_i$ is a boundary component of $X_{\geq 0}.$ Suppose that $C = Y_i$ is a boundary component. By the recursive definition of positive geometries, we have that $\omega_{(C,C_{\geq 0})} := \text{Res}_C(\omega_{(X,\sigma)})$ is the canonical form of $(C, C_{\geq 0}),$ and thus lies in $\Omega^{n-1}_{\log}(C^{\reg})$ by the inductive hypothesis. Since $C^{\reg}$ contains $C \cap Y^{\reg},$ we have that $\omega_{(C,C_{\geq 0})}$ lies in $\Omega^{n-1}_{\log}(C \cap Y^{\reg})$. Thus, $\omega_{(X,\sigma)}\in\Omega^n_{\log}(U)$.

    By the discussion surrounding Equation (4) in \cite{brown-dupont}, $\omega([\sigma])$ is the unique element of $\Omega^n_{\log}(X \setminus Y)$ satisfying the condition 
\begin{equation}\label{eq:rescondition}
    \rm{Res}(\omega([\sigma])) = \omega(\partial 
[\sigma]). 
\end{equation}
Suppose now that $(X,\sigma)$ is a positive geometry in the sense of \cite{Arkani_Hamed_2017}. Then its canonical form $\omega_{(X,\sigma)}$ also satisfies the residue condition \eqref{eq:rescondition}. Since $\omega_{(X,\sigma)}$ lies in $\Omega^n_{\log}(U),$ we have that $\omega_{(X,\sigma)}$ equals $\omega([\sigma])$ up to multiplication by a nonzero scalar.
\end{proof}

Finally, we recall the notion of \emph{real combinatorial rank} $\mathrm{cr}_\mathbb{R}(X,Y)$ of a positive arrangement $(X,Y)$ \cite[Def. 5.10]{cubicsurfaces}. It is defined as the dimension of the $\mathbb{Q}$-vector space spanned by the canonical forms of the real regions of $(X,Y)$. One always has $\mathrm{cr}_{\mathbb{R}}(X,Y)\leq \mathrm{cr}(X,Y)$. The situation in which the equality is attained is favorable from a computational perspective. In this case, the canonical forms of the real regions of $X\setminus Y$ yield a basis of $\Omega_{\log}^n(X\setminus Y)$. 

\begin{example}[Conics in the plane]
    Consider the arrangement of two conics in the projective plane pictured in Figure \ref{fig:twoconics}. 
    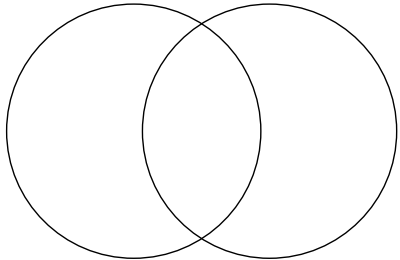
\begin{figure}[!h]
        \scalebox{0.7}{\tikzset{every picture/.style={line width=0.75pt}} 

\begin{tikzpicture}[x=0.75pt,y=0.75pt,yscale=-1,xscale=1]

\draw   (180,151.25) .. controls (180,100.85) and (220.85,60) .. (271.25,60) .. controls (321.65,60) and (362.5,100.85) .. (362.5,151.25) .. controls (362.5,201.65) and (321.65,242.5) .. (271.25,242.5) .. controls (220.85,242.5) and (180,201.65) .. (180,151.25) -- cycle ;
\draw   (277.5,151.25) .. controls (277.5,100.85) and (318.35,60) .. (368.75,60) .. controls (419.15,60) and (460,100.85) .. (460,151.25) .. controls (460,201.65) and (419.15,242.5) .. (368.75,242.5) .. controls (318.35,242.5) and (277.5,201.65) .. (277.5,151.25) -- cycle ;

\end{tikzpicture}}
        \caption{Two conics  in the projective plane}\label{fig:twoconics}
    \end{figure}
    In this case $X$ is $\bP^2$ and $Y$ is the union of the two conics. The divisor $Y$ has simple normal crossings, and each real region yields a positive geometry. The log canonical bundle is $\mathcal{O}_{\mathbb{P}^2}(K_{\bP^2}+H),$ where $H$ is the class of a hyperplane. However, the canonical forms of the three real regions are linearly dependent and do not span the full space of global sections $H^0(\bP^2, \mathcal{O}_{\mathbb{P}^2}(K_{\bP^2}+H)),$ which is three-dimensional. In other words, the real combinatorial rank of this pair does not equal its combinatorial rank, see \cite[Ex. 5.11]{cubicsurfaces}. 
\end{example}

\section{Log Canonical Models via Canonical Forms} \label{sec:main}
In this section we will show how to obtain log canonical embeddings of a large class of open varieties from canonical forms of their real regions. This is the technical heart of the paper. The following sections will focus on deriving explicit equations for these embeddings via positive geometries.

Determining the log canonical compactification of a given open variety  $U$ explicitly can be subtle, see e.g. \cite{KeelTevelevM0n, tevelev2006hyperplanes, HKTdelPezzo,corey2021initial,corey2023x38}. Even if one starts with a simple normal crossing compactification $X$ of $U$, to obtain the log canonical model one may have to apply a sequence of blowups and blow-downs. In the cases we consider, we obtain an embedding of the log canonical compactification directly by using canonical forms of real regions. We illustrate this with an example of line arrangement complements in the projective plane. 

\begin{example}\label{eg:lines}
    Figure \ref{fig:fivelines} shows three different arrangements $\cA_1, \cA_2,$ and $\cA_3$ of five lines $L_1, \ldots, L_5$ in $\bP^2.$ In each case, the wonderful compactification of $\bP^2 \setminus \cA_i$ is (by definition) the blowup at the points where at least three lines meet. This is a simple normal crossing compactification. 
    
    In the first case, the log canonical compactification equals $\rm{Bl}_p\bP^2$.
    In the second case, the bundle $\mathcal{O}_X(K_X + D)$ is not big (the Krull dimension of the log canonical ring is two).
    In the third case, the log canonical compactification is $\bP^1 \times \bP^1$ and is obtained from the blowup $\rm{Bl}_{p,q} \bP^2$ by contracting the proper transform of the line connecting $p$ and $q$;  see \cite[Ex. 4.1]{tevelev2007tori}.  
    \begin{figure}[!h]
        \scalebox{0.8}{\tikzset{every picture/.style={line width=0.75pt}} 

\begin{tikzpicture}[x=0.75pt,y=0.75pt,yscale=-1,xscale=1]

\draw    (631,180) -- (451,180) ;
\draw    (520,60) -- (620.75,210.75) ;
\draw    (560,60) -- (483.75,211.25) ;
\draw    (460,200) -- (620,120) ;
\draw    (600,180) -- (500,180) ;
\draw [shift={(500,180)}, rotate = 180] [color={rgb, 255:red, 0; green, 0; blue, 0 }  ][fill={rgb, 255:red, 0; green, 0; blue, 0 }  ][line width=0.75]      (0, 0) circle [x radius= 3.35, y radius= 3.35]   ;
\draw [shift={(600,180)}, rotate = 180] [color={rgb, 255:red, 0; green, 0; blue, 0 }  ][fill={rgb, 255:red, 0; green, 0; blue, 0 }  ][line width=0.75]      (0, 0) circle [x radius= 3.35, y radius= 3.35]   ;
\draw    (640,200) -- (480,120) ;
\draw    (416,181) -- (236,181) ;
\draw    (305,61) -- (405.75,211.75) ;
\draw    (345,61) -- (268.75,212.25) ;
\draw    (245,201) -- (405,121) ;
\draw    (385,181) -- (285,181) ;
\draw [shift={(285,181)}, rotate = 180] [color={rgb, 255:red, 0; green, 0; blue, 0 }  ][fill={rgb, 255:red, 0; green, 0; blue, 0 }  ][line width=0.75]      (0, 0) circle [x radius= 3.35, y radius= 3.35]   ;
\draw    (240,220) -- (380,100) ;
\draw    (200,180) -- (20,180) ;
\draw    (60,60) -- (169.75,210.75) ;
\draw    (109,60) -- (32.75,211.25) ;
\draw    (9,200) -- (169,120) ;
\draw    (149,180) -- (49,180) ;
\draw [shift={(49,180)}, rotate = 180] [color={rgb, 255:red, 0; green, 0; blue, 0 }  ][fill={rgb, 255:red, 0; green, 0; blue, 0 }  ][line width=0.75]      (0, 0) circle [x radius= 3.35, y radius= 3.35]   ;
\draw    (200,200) -- (20,100) ;

\draw (521,205) node   [align=left] {\begin{minipage}[lt]{27.2pt}\setlength\topsep{0pt}
$\displaystyle p$
\end{minipage}};
\draw (590,205) node   [align=left] {\begin{minipage}[lt]{13.6pt}\setlength\topsep{0pt}
$\displaystyle q$
\end{minipage}};
\draw (306,206) node   [align=left] {\begin{minipage}[lt]{27.2pt}\setlength\topsep{0pt}
$\displaystyle p$
\end{minipage}};
\draw (70,205) node   [align=left] {\begin{minipage}[lt]{27.2pt}\setlength\topsep{0pt}
$\displaystyle p$
\end{minipage}};
\draw (546,159) node [anchor=north west][inner sep=0.75pt]   [align=left] {$\displaystyle 1$};
\draw (540,115) node [anchor=north west][inner sep=0.75pt]   [align=left] {$\displaystyle 4$};
\draw (517,147) node [anchor=north west][inner sep=0.75pt]   [align=left] {$\displaystyle 2$};
\draw (567,145) node [anchor=north west][inner sep=0.75pt]   [align=left] {$\displaystyle 3$};

\end{tikzpicture}}
        \caption{Three arrangements $\cA_1, \cA_2, \cA_3$ of five lines in $\bP^2$}\label{fig:fivelines}
    \end{figure}
    
    One can verify the claims for $\mathcal{A}_1$ and $\mathcal{A}_3$ by using canonical forms. More precisely, we will show in Theorem \ref{thm:hyperplanearr} that the log canonical embeddings coincide
with the Zariski closures of the maps given by the canonical forms of the real bounded regions of $\mathbb{P}^2\setminus\mathcal{A}_i$. For example, for the arrangement $\cA_3$, we choose the concrete parameterization given in Example \ref{eg:lines prelim}. The bounded regions are labeled one through four. The canonical forms of the triangles $1, 24, 34,$ and $1234$ yield the map 
    \begin{align*}
        \bP^2 \setminus \cA_3 & \to \bP^3, \\ 
        [x: y: z] & \mapsto \left[\frac{1}{z(x-z)(y-z)}: \frac{1}{yz(x-z)}:\frac{1}{xz(y-z)}:\frac{1}{x y z}\right].
    \end{align*}
    Canonical forms are additive along triangulations of polytopes \cite[Section 3.1]{Arkani_Hamed_2017}. Thus the canonical forms the four triangles span the same space as those of the bounded regions. 
    The image has Zariski closure $\bP^1 \times \bP^1,$ given by the vanishing of a single quadric.
\end{example}

In the remainder of this section, we explain and generalize the phenomenon illustrated in Example \ref{eg:lines}. Suppose that the pair $(X,Y)$ has genus zero. Then the canonical form map~\eqref{eq:canonical map} is surjective by~\cite[Remark~2.7]{brown-dupont}, so the canonical forms associated to a basis of $H_n(X,Y)$ span $\Omega^n_{\log}(U)$. The linear subspace of $|K_X+Y|$ of forms whose restriction to $U$ lies in $\Omega^n_{\log}(U)$ yields a rational map 
\[
\phi_{\mathrm{can}} \colon X \dashrightarrow \mathbb{P}\Omega^n_{\log}(U)^\vee.
\]
Choose a basis $\omega_1,\ldots,\omega_s$ of $\Omega^n_{\log}(U)$ consisting of canonical forms.
Restricting to $U=X\setminus Y$, where the forms are regular, this map is given concretely by
\begin{equation}
\phi_{\mathrm{can}} \colon U \longrightarrow \mathbb{P}\Omega^n_{\log}(U)^\vee,
\qquad
u \longmapsto [\omega_1(u):\cdots:\omega_s(u)].
\end{equation}

Our goal is to describe the image of $\phi_{\mathrm{can}}$. Under suitable hypotheses, this image is precisely the log canonical compactification of~$U$.

\begin{theorem} \label{thm:mainnew}
    Let $U$ be a smooth complex algebraic variety whose log canonical compactification $X_{\rm lc}$ exists. Let $X$ be a compactification of $U$ with boundary divisor $Y:=X\setminus U$.
    \begin{enumerate}
        \item Suppose that the pair $(X,Y)$ has genus zero and the log canonical bundle $K_{X_{\rm lc}} + D_{\rm lc}$ is very ample. Then the map $\phi_{\mathrm{can}}:U\to\mathbb{P}^{\mathrm{cr}(X,Y)-1}$ equals the log canonical embedding of $U$. The image of this map does not depend on the choice of compactification.
        \item Suppose further that $(X,Y)$ is a positive arrangement with $\mathrm{cr}_\mathbb{R}(X,Y)=\mathrm{cr}(X,Y)$. Then the canonical forms of the real regions realize this log canonical embedding. 
    \end{enumerate}
\end{theorem}

\begin{proof}
    Assume $X$ is an SNC compactification of $U$. Let $R(U)$ be the log canonical ring of $U.$ Then by \cite[Proposition 1.13]{brown-dupont} we have the isomorphism $\Omega^n_{\mathrm{log}}(U)\cong H^0(X,K_X+Y)$, so the space of logarithmic forms is isomorphic to the degree one part of $R(U),$ which is precisely $H^0(X_{\rm lc},K_{X_{\rm lc}}+D_{\rm lc}).$
    
    Since the genus $g(X,Y)$ is zero, we obtain the surjective map $\omega: H_n(X, Y) \to \Omega^n_{\mathrm{log}}(U)$ as in \eqref{eq:canonical map}, which sends each homology cycle to its canonical form. This gives a spanning set of sections of $H^0(X_{\rm lc},K_{X_{\rm lc}}+D_{\rm lc}).$ Since the bundle $K_{X_{\rm lc}}+D_{\rm lc}$ is very ample, the sections give an embedding of $X_{\rm lc}$ in projective space, which restrict to the map $\phi_{\rm{can}}$ on the open set $U.$ Since polynomial maps are closed and $U$ is dense in $X_{\rm lc},$ we have that $\overline{\phi(U)}$ is the log canonical compactification and $\phi$ is the log canonical embedding. 

    Now suppose $X$ is not an SNC compactification. By Hironaka's resolution of singularities \cite{hironaka}, there exists a modification $(\tilde{X},\tilde{Y})$ in the sense of \cite[Def. 1.2]{brown-dupont} that is an SNC compactification. The space of canonical forms is invariant under modifications \cite[Section 2.3.1]{brown-dupont}. That is, the pushforward along $f: \tilde{X} \to X$ yields an isomorphism of relative homology groups $f_*: H_n(\tilde{X}, \tilde{Y}) \to H_n(X,Y)$ such that $\omega^X \circ f_* = \omega^{\tilde{X}}.$  Thus the images of $\omega^X$ and $\omega^{\tilde{X}}$ both span the vector space $\Omega^n_{\log}(U).$ 
    
    Finally, if $\mathrm{cr}_{\mathbb{R}}(X,Y)=\mathrm{cr}(X,Y)$, then the canonical forms of the real regions of $U$ already span the space $\Omega^n_{\mathrm{log}}(U)$ and yield the desired embedding.
\end{proof}

For the examples we consider -- hyperplane arrangement complements in Section \ref{sec:hyperplanes} and nonlinear examples in Section \ref{sec:nonlinear} -- the existence of the log canonical compactifications has been proven by other sources. We thus concentrate on finding the equations of these embeddings.

The situation in which all real regions are positive geometries and $\mathrm{cr}_{\mathbb{R}}(X,Y)=\mathrm{cr}(X,Y)$ is particularly nice, since in this case the canonical forms can be written down by following a combinatorial procedure. This yields an explicit parametrization of the log canonical embedding of $U$ and allows us to find its defining equations by using implicitization techniques. This is also the case in all examples that we consider in the following sections.

\section{Equations for Compactifications of Hyperplane Arrangement Complements}\label{sec:hyperplanes}
Complements of hyperplane arrangements are one of the simplest classes of open varieties satisfying the conditions of Theorem \ref{thm:mainnew}. Their log canonical models were described by Hacking, Keel and Tevelev in \cite[Sec. 2]{tevelev2006hyperplanes}. In this section, we recall this construction, show how to obtain it via canonical forms, and present an algorithm to find its defining equations.

We start by reviewing the construction in \cite[Sec. 2]{tevelev2006hyperplanes}. Consider an essential arrangement $\mathcal{A}$ of $m$ complex hyperplanes defined over $\mathbb{R}$ inside the projective space $\bP^{n-1}$. Assume the underlying matroid of $\mathcal{A}$ is connected. We define $U := \bP^{n-1} \setminus \mathcal{A}$ to be the complement. We denote the real linear forms defining this arrangement by $F_1, \ldots, F_m.$ Then there is a map
\begin{align*}
    F: \bP^{n-1} & \to \bP^{m-1}, \\
    x& \mapsto [F_1(x): \cdots : F_m(x)],
\end{align*}
which embeds $\bP^{n-1}$ as a linear subspace of $\bP^{m-1}.$ Thus $F$ determines a point $L$ in the Grassmannian $\Gr(n,m)$ of $(n-1)$-dimensional projective subspaces of $\bP^{m-1}.$ Let $G_e$ denote the sub-Grassmannian of $\Gr(n,m)$ parameterizing linear spaces which contain the vector $e := (1,1,\ldots, 1).$ We intersect $G_e$ with the torus orbit closure of $L$ to obtain
\begin{equation}
    S := \overline{(\bC^*)^m \cdot L} \, \cap \, G_e.
\end{equation}
The variety $U$ is embedded into $S$ as an open subvariety as follows. The map $F$ sends $U$ to a closed subvariety of the torus $(\bC^*)^m.$ Then the map 
\begin{align*}
    \psi: U & \to \Gr(n,m), \\
    u & \mapsto F(u)^{-1} \cdot L
\end{align*}
gives an embedding of $U$ into $S.$ Let $B:= S \setminus \psi(U).$ The pair $(S,B)$ is called the \emph{visible contour} of $U$ and was defined by Kapranov \cite{kapranov1993chow}. Then \cite[Theorem 2.2]{tevelev2006hyperplanes} proves that $(S,B)$ is the log canonical compactification of $U$, and moreover that $K_S +  B$ is very ample. Our first result is that this compactification is determined by quadrics scheme-theoretically. 

\begin{proposition}
    Let $\mathcal{A}$ be an essential hyperplane arrangement with connected underlying matroid. The Hacking--Keel--Tevelev log canonical compactification $S$ of $\mathbb{P}^{n-1}\setminus \mathcal{A}$ is defined by quadrics as a scheme.
\end{proposition}

\begin{proof}
    By \cite[Lemma 2.1]{tevelev2006hyperplanes}, $S$ is equal to the scheme-theoretic intersection of $G_e$ and the toric variety $T$ associated with the matroid of $\mathcal{A}$. The Grassmannian $G_e$ is cut out by quadrics, and by \cite[Theorem 3]{michalek2014white} the same is true for $T$ up to saturation, which yields the claim. 
\end{proof}

 By composing the map $\psi$ with the Pl\"ucker embedding of $G_e,$ we obtain a map $\phi$ from $U$ to $\bP^{\binom{m-1}{n-1}-1}.$ This is the log canonical embedding of $U$ \cite[Theorem 2.2]{tevelev2006hyperplanes}.
 
We now explain how to view the embedding $\phi$ in coordinates.
We choose coordinates $(z_0:\cdots:z_{n-1})$ on $\bP^{n-1}$ so that $F_m = z_{n-1}$. This allows us to use local coordinates $x_i := z_i/z_{n-1}, \ i=0, \, \ldots, \, n-2$ on an affine chart in $\bP^{n-1}$. We then define $\omega_i := d\log (F_i / F_m)$ for $i=1,\ldots,m-1$.  This can be written uniquely as 
\[\omega_i=\sum_{j=0}^{n-2} a_{i,j}(x_0, \, \ldots, \, x_{n-2})dx_j,\]
where $a_{i,j}$ are rational functions.
Then in the basis $\{dx_i\}_{i=0}^{n-2}$, the map 
\[\bC^{m-1} \otimes \mathcal{O}_U \to \Omega_U\]
is given concretely by the $(n-1) \times (m-1)$ matrix 
$$M(x_0,\ldots,x_{n-2}) = \begin{bmatrix}
    a_{1,0} & \ldots & a_{m-1,0} \\
    \vdots & \ddots& \vdots\\
    a_{1,n-2} &\ldots & a_{m-1,n-2}
\end{bmatrix}.$$
Composing the map $\mathbb{P}^{n-1}\setminus \mathcal{A} \to \mathrm{Gr}(n-1,m-1)$ given by $M$ with the Pl\"ucker embedding gives an embedding into $\bP^{\binom{m-1}{n-1}-1}$. 

\begin{remark}
    The ambient space of the embedding described above is $\mathbb{P}^{\binom{m-1}{n-1}-1}$. The number of coordinates is the number of bounded regions in a simple arrangement of $m$ hyperplanes in $\mathbb{P}^{n-1}$ with a generic hyperplane at infinity \cite{zaslavsky}. When the arrangement is not simple, the number of bounded regions drops, and we will see in Theorem \ref{thm:hyperplanearr} that the embedding lives in a proper projective subspace of $\mathbb{P}^{\binom{m-1}{n-1}-1}$. 
\end{remark}

Hyperplane arrangement complements satisfy the conditions of Theorem \ref{thm:mainnew}. Moreover, in this special case to obtain the log canonical embedding, it suffices to only consider canonical forms of regions that are bounded in some generic affine chart. Here ``generic'' means that the chart contains all of the intersections of the hyperplanes of the arrangement. 

\begin{theorem}\label{thm:hyperplanearr}
    Let $\mathcal{A}$ be an essential arrangement of $m$ hyperplanes in $\pp^{n-1}$ defined over $\mathbb{R}$ with complement $U := \pp^{n-1} \setminus \mathcal{A}.$ Suppose that the matroid of $\cA$ is connected. Pick a generic affine chart of $\mathbb{P}^{n-1}$ and let $\omega_1, \, \ldots, \, \omega_r \in H^0(\pp^{n-1}, \mathcal{O}_{\mathbb{P}^{n-1}}(K_{\pp^{n-1}} + m H))$ be the canonical forms of the real regions of $U$ that are bounded in this chart. Then the map
    \begin{align*}
        U & \to \pp \Omega^{n-1}_{\log}(U)^\vee, \\
        u & \mapsto [\omega_1(u)  : \cdots :\omega_r(u)]
    \end{align*}
    equals the log canonical embedding of $U$.
\end{theorem}
\begin{proof}
    If the hyperplane arrangement is essential and its underlying matroid is connected, then the log canonical compactification of $U$ is given by Kapranov's visible contours and the log canonical divisor $K_S + B$ is very ample \cite[Theorem 1.5]{tevelev2007tori}. 
    Furthermore, the bounded real regions of $U$ are all polytopes, and hence positive geometries. By Theorem \ref{thm:mainnew}, it suffices to show that the span of the canonical forms of these regions equals $\Omega^{n-1}_{\log}(U)$ rather than some proper subspace. But by \cite{brieskorn} and \cite[Remark 2.8]{brown-dupont}, the space $\Omega_{\log}^{n-1}(U)$ is isomorphic to $H_{n-1}(\bP^{n-1} , \cA)$. The latter space is spanned by the classes of the bounded real regions of $\bP^{n-1}_{\bR} \setminus \cA_{\bR}$ \cite[Proposition 6.2]{brown-dupont}. 
\end{proof}

In many cases, the log canonical compactification of $\pp^{n-1} \setminus \mathcal{A}$ equals the wonderful compactification. Combinatorial conditions for these two to be equal were found in \cite{Feichtner2005}. An example where the two compactifications differ is given by Example \ref{eg:lines}, or more generally  by a line in the plane and two points on it, such that every other line meets one of these two points \cite[Example 4.1]{tevelev2007tori}. In that case, the wonderful compactification is the blowup of $\bP^2$ at two points, and the log canonical compactification $\bP^1 \times \bP^1$ is obtained by contracting a curve. 

\begin{example}[Generic hyperplanes]
Let $H_1, \ldots, H_m$ be hyperplanes in $\pp^{n-1}$ in general position. Then $\pp^{n-1}$ itself is a simple normal crossing compactification. The divisor $D$ given by the union of these hyperplanes is rationally equivalent to $mH,$ where $H$ is the divisor of a hyperplane. The canonical bundle of $\mathbb{P}^{n-1}$ is $\co_{\pp^{n-1}}(-n).$ Thus $\Omega_X^{n-1}(D) = \co_{\pp^{n-1}}(m-n),$ which is ample whenever $m$ is greater than $n.$ In this case, the embedding given by the canonical forms is linearly isomorphic to the Veronese embedding $\nu_{m-n}(\mathbb{P}^{n-1})$. 
\end{example}

\begin{example}
    Let us consider the hyperplane arrangement $\cA_3$ from Example \ref{eg:lines}. We set $F_5 = z$ and use the local coordinate system $u: = x/z$ and $v:= y/z.$ We write
\[\omega = \begin{bmatrix}
u & u-1 & 0 & 0 \\
0 & 0 & v & v-1
\end{bmatrix}^T \begin{bmatrix}
    du \\ dv
\end{bmatrix}.\] 
The four columns represent the forms $\omega_i.$ The image under the Pl\"ucker embedding lies in $H^0(\bP^2, \Omega^2_{\bP^2}(2H))$ which is six-dimensional. It lies in the subspace spanned by the four forms $\omega_1 \wedge \omega_3, \omega_2 \wedge \omega_3, \omega_1 \wedge \omega_4,$ and $\omega_2 \wedge \omega_4,$ which satisfy the single quadratic relation $(\omega_1 \wedge \omega_3) \cdot (\omega_2 \wedge \omega_4) = (\omega_1 \wedge \omega_4) \cdot (\omega_2 \wedge \omega_3)$ coming from the Pl\"ucker relation of $\Gr(2,4).$ Thus we also recover the log canonical compactification $\bP^1 \times \bP^1$ using the approach of \cite{tevelev2006hyperplanes}.
\end{example}

Canonical forms of polytopes have a clear combinatorial structure. They have the form 
$$\dfrac{\alpha_P(x)}{f_1(x)\ldots f_k(x)}\Omega_{\mathbb{P}^{n-1}}(x),$$
where $\Omega_{\mathbb{P}^{n-1}}(x) = \sum\limits_{i=0}^{n-1} (-1)^{i} x_i dx_0\wedge \ldots \wedge \widehat{dx_i}\wedge \ldots \wedge dx_{n-1}$, the $f_i(x)$ are the linear forms defining the facets of the polytope, and $\alpha_P(x)$ is the \emph{adjoint polynomial} of $P$, see \cite[Sec. 1]{kohnranestad}. A concrete recipe for computing the adjoint polynomial from the defining equations of the facets for arbitrary polytopes was given in \cite{clemensjulian}. Based on this, we present Algorithm \ref{alg:hyp} for computing the ideal of the log canonical embedding of $\mathbb{P}^{n-1}\setminus \mathcal{A}$ for an essential hyperplane arrangement $\mathcal{A}$ whose underlying matroid is connected. 
We implemented Algorithm \ref{alg:hyp} in \texttt{Julia}. The implementation along with additional comments is available at \cite{zenodo}. 

We conclude this section with a series of examples illustrating our approach in action. We consider a generic hyperplane arrangement and a non-generic one, and study how the log canonical compactification degenerates from one to the other.

\begin{algorithm}[h!]
\caption{Computing the log canonical model of a hyperplane arrangement complement}\label{alg:hyp}
\textbf{Input:}  An $m\times n$ matrix $A$ defining a hyperplane arrangement $\mathcal
A$ in $\mathbb{P}^{n-1}$\\
\textbf{Output:} The ideal of the log canonical embedding of $\mathbb{P}^{n-1}\setminus \mathcal{A}$ in $\mathbb{P}^{\mathrm{cr}(\mathbb{P}^{n-1},\mathcal{A})-1}$ 

\begin{algorithmic}[1]
\State\label{step:1}Identify the regions of $\mathcal{A}$ and label them with sign vectors. 

\State\label{step:2}For each region $P$, identify its supporting hyperplanes using the sign vectors and build a matrix $A_P$ whose rows define these hyperplanes. 

\State\label{step:3}From $A_P$, compute the \emph{point residual arrangement} $\mathcal{R}_0(P)$ of each region by using \cite[Section 6]{clemensjulian} along with the multiplicity of each point.

\State\label{step:4}Compute the unique polynomial $\alpha_P$ of minimal degree interpolating $\mathcal{R}_0(P)$. This is the adjoint of $P$.

\State\label{step:5}For each region, construct the \emph{canonical function} $\alpha_P/(f_{1_P}\cdots f_{k_P})$, where $f_{i_P}$ are the linear forms defining the facets of $P$. 

\State\label{step:6}Compute a basis of the linear space spanned by the canonical functions.

\State\label{step:7}Implicitize the variety parametrized by this basis to obtain its defining ideal $J$.

\Return \label{step:return} $\text{A set of generators of } J$, the ideal of the log canonical embedding
\end{algorithmic}
\end{algorithm}

\newpage

\begin{example}[A generic case]
    We work in the affine chart of $\bP^2$ parameterized by $[x:y:1].$ Consider the hyperplane arrangement defined by 
\[
f_1 = x, \quad
f_2 = y, \quad
f_3 = 1, \quad
f_4 = x+y+1, \quad
f_5 = x+2y+3.
\]

We first describe the log canonical compactification of its complement following \cite{tevelev2006hyperplanes}. Define logarithmic $1$-forms
\[
\omega_i = d\log\!\left(\frac{f_i}{f_5}\right), \quad i=1,\dots,4.
\]
Writing $\omega_i = a_{i,0}\,dx + a_{i, 1}\,dy$, we obtain
\[
M(x,y)=
\begin{pmatrix}
\dfrac{2y+3}{x(x+2y+3)} &
-\dfrac{1}{x+2y+3} &
-\dfrac{1}{x+2y+3} &
\dfrac{y+2}{(x+y+1)(x+2y+3)}
\\[1em]
-\dfrac{2}{x+2y+3} &
\dfrac{x+3}{y(x+2y+3)} &
-\dfrac{2}{x+2y+3} &
\dfrac{1-x}{(x+y+1)(x+2y+3)}
\end{pmatrix}.
\]

The induced map $U \to \Gr(2,4) \subset \mathbb{P}^5$ is given by the $2\times2$ minors of $M(x,y)$.
After clearing denominators, we obtain the embedding:
\begin{equation}\label{eq:hkt generic}
(x,y) \longmapsto
\bigl[
3(x+y+1):
-2y(x+y+1):
y:
x(x+y+1):
-2x:
xy
\bigr].
\end{equation}
Up to a linear change of coordinates, the image closure is the Veronese embedding of $\bP^2.$

We now show that this is linearly isomorphic to the variety parametrized by the canonical forms. Let us work in the affine chart with coordinates $(u,v)$ parameterized by $[u:v:1-2u-3v]$. The defining functions in this chart become
\[
\ell_1=u,\qquad
\ell_2=v,\qquad
\ell_3=1-2u-3v,\qquad
\ell_4=1-u-2v,\qquad
\ell_5=3-5u-7v.
\]
There are six bounded regions $P_i$, with canonical functions given by:
{\small\[
\begin{aligned}
F_0(u,v) &= \frac{1}{u\,v\,(1-2u-3v)},
&
F_1(u,v) &= \frac{u+2v}
{u\,v\,(2u+3v-1)\,(3-5u-7v)},\\[1ex]
F_2(u,v) &= \frac{1}
{(-u)\,(1-u-2v)\,(5u+7v-3)},
&
F_3(u,v) &= \frac{2u+v}
{u\,v\,(1-u-2v)\,(5u+7v-3)},\\[1ex]
F_4(u,v) &= \frac{2-3u-4v}
{(-u)\,(2u+3v-1)\,(1-u-2v)\,(3-5u-7v)},
&
F_5(u,v) &= \frac{1}
{(-v)\,(2u+3v-1)\,(3-5u-7v)}.
\end{aligned}
\]}
Now, we change to the coordinates $x$ and $y.$ Let
\[
D:=x\,y\,(x+y+1)\,(x+2y+3).
\]
Then the canonical functions of $P_i$ may be written in the form $N_i(x,y)/D.$
The functions $N_i$ then give the log canonical embedding
\[
\bigl[
x^2+3xy+4x+2y^2+5y+3:
-(x^2+3xy+x+2y^2+2y):
y:
-(2x+y):
xy+2y^2+2y:
x^2+xy+x
\bigr].
\]
The two parametrizations are indeed related by a linear change of coordinates:
\[
\begin{pmatrix}
N_0\\ N_1\\ N_2\\ N_3\\ N_4\\ N_5
\end{pmatrix}
=
\begin{pmatrix}
1&-1&0&1&0&0\\
0&1&0&-1&0&0\\
0&0&1&0&0&0\\
0&0&-1&0&1&0\\
0&-1&0&0&0&-1\\
0&0&0&1&0&0
\end{pmatrix}
\begin{pmatrix}
G_0\\ G_1\\ G_2\\ G_3\\ G_4\\ G_5
\end{pmatrix},
\]
where $G_i$ are the coordinate functions given by the embedding in \eqref{eq:hkt generic}.
\end{example}

\begin{example}[A non-generic case]
As before, we work in the affine chart of $\bP^2$ parameterized by $[x:y:1].$ Consider the hyperplane arrangement defined by the affine-linear functions
\[
f_1 = x, \quad
f_2 = y, \quad
f_3 = -2x-3y, \quad
f_4 = x+y+1, \quad
f_5 = x+2y+3. \quad
\]
This arrangement is not simple: the hyperplanes defined by $f_1,f_2,f_3$ intersect at the origin. Note that this arrangement is a degeneration of the arrangement in the previous example where the hyperplane given by $f_3$ has been shifted to intersect the origin in the affine chart with coordinates $(u, v)$ as described above. 

As before, we define logarithmic $1$-forms
\[
\omega_i = d\log\!\left(\frac{f_i}{f_5}\right), \quad i=1,\dots,4.
\]
The construction of \cite{tevelev2006hyperplanes} with these forms produces the following map $\mathbb{P}^2\dashrightarrow \mathbb{P}^4$
\begin{equation*}\label{eq:hkt nongeneric}
\bigl[
-3xy-6y^2-9y: 
2x^2+4xy+6x:
-9xy-9y^2-9y:
6x^2+6xy+6x:
-4xy
\bigr].
\end{equation*}
In this non-generic arrangement, there are only five bounded regions, as opposed to six in the generic case. After clearing of denominators, the canonical forms give the map
\begin{equation*}
\bigl[
7xy+6y^2+9y:
2xy+3y^2:
x^2+2xy+3x:
2x^2+7xy+6y^2+6x+9y:
4x^2+8xy+3y^2
\bigr].
\end{equation*}
Just as in the generic case, there is a linear change of coordinates between the coordinate functions $G_i$ of Hacking--Keel--Tevelev, and the canonical forms $N_i$ (after putting everything over a common denominator). In this case, the coordinate change is given by 
\[
\begin{pmatrix}
N_0\\ N_1\\ N_2\\ N_3\\ N_4
\end{pmatrix}
=
\begin{pmatrix}
3 & 0 & 0 & 0 & 4 \\
-3 & 0 & 9 & 0 & -4 \\
0 & 1 & 0 & 0 & 0 \\
3 & 2 & 0 & 0 & 0 \\
-3 & -2 & 9 & 6 & 0
\end{pmatrix}
\begin{pmatrix}
G_0\\ G_1\\ G_2\\ G_3\\ G_4
\end{pmatrix}. 
\]

\end{example}

\begin{example}[Degenerations of a hyperplane arrangement]
    In this example we examine what happens to the log canonical compactification when we degenerate from five hyperplanes in general position to a situation in which three of the hyperplanes meet at a point. For example, consider the family of hyperplanes defined by the functions
    \[\mathcal{A}_t = \{x,y,z,x+y+z,x + 2y + tz\},\]
    where $t\in \mathbb{C}$ is a parameter.
    This yields a family $U_t$ of points in $\Gr(3,5).$ We define the ideal \[I_t := I(\overline{(\bC^*)^5 \cdot U_t} \cap G_e).\]
    For a generic parameter $t$, the ideal $I_t$ defines the log canonical compactification of the hyperplane arrangement complement. The fiber above $t=0$ splits into two components, which are isomorphic to $Bl_p \bP^2$ and a copy $P$ of the projective plane. They intersect at the exceptional divisor in $Bl_p \bP^2,$ which is a line in the plane $P.$ The degrees of these components are $3$ and $1,$ respectively. The first component is the log canonical compactification of the hyperplane arrangement with $t=0$.
\end{example}

\section{Nonlinear Examples: Cubic Surfaces and Configuration Spaces}\label{sec:nonlinear}

\subsection{Positive geometry of the 27 lines on the cubic surface}
Smooth cubic surfaces in $\mathbb{P}^3$ are among the most classical and well-studied objects in algebraic geometry. Each such surface contains 27 lines. When all the lines are real, the real regions of the complement of these lines are positive geometries \cite{cubicsurfaces}. Here we explicitly compute their canonical forms and the corresponding log canonical compactification. We first comment on why these canonical forms produce the log canonical compactification.

Let $X\subset \mathbb{P}^3$ be a smooth cubic surface all of whose lines are real and $Y$ be the divisor of $27$ lines on it.
The pair $(X,Y)$ has genus zero: $X$ is smooth and rational, and thus has genus zero itself, and the intersection strata of $Y$ are either lines or points, all of which have genus zero. 
The claim now follows from \cite[Corollary 3.13]{brown-dupont}. 
The divisor $Y=9(-K_X)$ is a positive multiple of the anticanonical divisor $-K_X$ and is thus very ample, since $X$ is a del Pezzo surface. 
The corresponding log canonical ring is $R=\bigoplus\limits_{m\geq 0} H^0(X, \mathcal{O}_X(8m))$ and is a Veronese subring of the coordinate ring of $X$. 
We have $\mathrm{cr}(X,Y)= \binom{8+3}{3} - \binom{8}{3}= 109 $ and a direct computation shows that the canonical forms of the real regions span a $109$-dimensional space and parametrize the eighth Veronese embedding of~$X$, the log canonical embedding of $X\setminus Y$. 

We now review the construction of canonical forms of the real regions of $X\setminus Y$ presented in \cite{cubicsurfaces}. 
Any smooth cubic surface is the blowup of $\mathbb{P}^2$ at six points in general position (no three on a line, all six not coconic). 
The 27 lines come in three flavors: $6$ lines $E_i$ which are the components of the exceptional divisor of this blowup, $15$ lines $F_{ij}$ which are the images of the lines through points $i$ and $j$ under this blowup, and $6$ lines $G_i$ which are the images of the conics through all points except $i$. Each line on the cubic surface is incident to ten lines, and the incidence relations are recorded by the graph $\mathcal{S}_{27}^{10}$, the complement of the Schl\"afli graph. 

Endow $\mathbb{P}^3$ with homogeneous coordinates $[y_0:y_1:y_2:y_3]$. Define
$$\Omega_{\mathbb{P}^3} := y_0dy_1\wedge dy_2 \wedge dy_3-y_1dy_0\wedge dy_2 \wedge dy_3 + y_2dy_0\wedge dy_1 \wedge dy_3 - y_3dy_0\wedge dy_1\wedge dy_2.$$
Now, a differential $3$-form on $\mathbb{P}^3$ has the form 
$$F(y_0,y_1,y_2,y_3)\Omega_{\mathbb{P}^3},$$
where $F$ is a rational function of degree $-4$. Suppose $X$ is defined by the equation $f(y_0,y_1,y_2,y_3)=0$ and we have the differential form 
\begin{equation} \label{eq:cf}
   \omega:=\dfrac{h(y)}{f(y)g(y)}\Omega_{\mathbb{P}^3}. 
\end{equation}
Then in the affine chart $y_0=1$ the residue of $\omega$ along $X$ is 
$$
\mathrm{Res}_X \omega=\dfrac{h(1,y_1,y_2,y_3)}{\dfrac{\partial f}{\partial y_3}(1,y_1,y_2,y_3)g(1,y_1,y_2,y_3)}dy_1\wedge dy_2.
$$
We will obtain the canonical forms of the regions of $X\setminus Y$ by taking residues of forms as in \eqref{eq:cf}. The corresponding canonical functions (i.e. the coefficients of $dy_1 \wedge dy_2$ in the canonical form) defining the log canonical embedding are the rational function coefficients of these residues homogenized by $y_i\mapsto y_i/y_0$ for $i=1,2,3$. 

The regions of $X\setminus Y$ are ``puffed'' triangles, quadrilaterals and pentagons, i.e. their boundaries are unions of at most five lines. These canonical forms were derived in \cite{cubicsurfaces}, and we provide more details on their construction. We then use them to compute the log canonical embedding explicitly.

For a triangle $L$ the canonical form is 
$\mathrm{Res}_{X} \dfrac{1}{f(y)h(y)}\Omega_{\mathbb{P}^3}$, where $h$ is the linear form defining the plane in $\pp^3$ containing the triangle.  For quadrilaterals, the numerator of the canonical function is a linear form. To define it, we identify the ``adjoint vertex'' of the quadrilateral in the graph $\mathcal{S}_{27}^{10}$. This is the unique line on $X$ not intersecting any of the four facet lines. Then there are exactly two ways to add two lines to the four facets of the quadrilateral to get two triangles $T_1$ and $T_2$, and complete the adjoint line to a triangle $T_3$. All these triangles correspond to $3$-cycles in $\mathcal{S}_{27}^{10}$.
The canonical form of the quadrilateral is then 
$$\mathrm{Res}_X \dfrac{h_3}{fh_{1}h_2}\Omega_{\mathbb{P}^3}, $$
where $h_i$ is the linear form defining the triangle $T_i$. Both choices of lines to construct the triangles $T_1,T_2$ and $T_3$ yield the same form on $X$. For examples, see \cite[p.~12]{cubicsurfaces}.

The interpretation for pentagons is similarly combinatorial. Each pentagon corresponds to a 5-cycle in the graph $\mathcal{S}_{27}^{10}$ (the vertices in this graph are the 27 lines, and edges connect lines that intersect). There is a unique 3-cycle $H_1$ on the vertices of this 5-cycle. In this 3-cycle there is a unique edge that does not correspond to a vertex of the pentagon, denote it $e_1$. The remaining two vertices of the 5-cycle that are not in $H_1$ can be uniquely completed to a 3-cycle $H_2$. Let $A$ denote the line corresponding to the vertex added to get this 3-cycle. Now let $g$ be the linear form vanishing on $A$ and the point corresponding to the edge $e_1$, and $h_1,h_2$ be the linear forms defining the planes corresponding to $H_1,H_2$, respectively. Then the canonical form of the pentagon is 
$$
\mathrm{Res}_X\dfrac{g}{h_1h_2f}\Omega_{\mathbb{P}^3}.
$$

We compute the canonical forms of the real regions of any smooth cubic surface $X$ and the corresponding log canonical embedding. The ideal of this embedding is generated by $5586$ quadratic polynomials, which we explicitly compute. We performed this computation in \texttt{Macaulay2}. The code for this can be found in our supplementary file \texttt{CubicSurface.m2} \cite{zenodo} and allows the user to carry out this computation for any six points of interest. 

\subsection{The configuration spaces $\mathcal{X}(3,6)$, $\mathcal{Y}(3,6)$ and $\mathcal{X}(3,7)$}

The goal of this section is to explicitly describe a rational parametrization and quadratic equations of the log canonical compactification of the configuration space $\mathcal{Y}(3,6)$, as well as to offer some perspective for the spaces $\mathcal{X}(3,6)$ and $\mathcal{X}(3,7)$. We start by defining relevant terminology.

The open configuration space of $n$ labeled points in $\mathbb{P}^{k-1}$ in general position is denoted by $\mathcal{X}(k,n)$. Let $\mathrm{Gr}^\circ(k,n)$ denote the open subset of $\mathrm{Gr}(k,n)$ where all Pl\"ucker coordinates are nonzero.
We define $\mathcal{X}(k,n)$ as the quotient $\mathrm{Gr}^\circ(k,n)/(\mathbb{C}^*)^n$ by the action of the algebraic torus $(\mathbb{C}^*)^n$ scaling the columns. 
When $k=3$ and $n=6,7$, the space $\mathcal{X}(3,n)$ is closely related to the moduli space $\mathcal{Y}(3,n)$ of marked del Pezzo surfaces of degree $9-n$. 
The latter is obtained from the former by removing $\binom{n}{6}$ coconic divisors, as described in \cite[Section 1]{posdelPezzo}.

Compactifications of $\mathcal{X}(3,n)$ have been well studied, and extensive research has been done comparing them, see \cite{luxton2008thesis,corey2021initial,corey2023x38,SchafflerTevelev}. Among these are the log canonical compactification and (the normalization of) Kapranov's Chow quotient. It was shown in \cite{KeelTevelevChow} that these compactifications differ for $n \geq 9$ and conjectured that for $n = 6, 7, 8$ they agree. This conjecture was proven for $n=6$ in \cite{luxton2008thesis}, for $n =7$ in \cite[Theorem 1.3]{corey2021initial}, and for $n=8$ in \cite[Theorem 1.1]{corey2023x38}. Finally, the log canonical compactifications of $\mathcal{X}(3,6)$ and $\mathcal{X}(3,7)$ are both tropical compactifications \cite{luxton2008thesis, corey2021initial}. We will exploit this fact to compute the combinatorial rank. 

\begin{definition}
    Let $U \subset \bC^n$ be a very affine variety. Let $\Sigma \subset \bR^n$ be a fan with $|\Sigma| = \rm{Trop} \, U.$ Let $\overline{U}$ be the corresponding tropical compactification of $U,$ namely the closure of $U$ in $X_\Sigma.$ Then $U$ is called \emph{sch\"on} if, for every torus orbit $O_\sigma,$ the intersection $\overline{U} \cap O_\sigma$ is smooth. If additionally all positive-dimensional strata $\overline{U} \cap O_\sigma$ are connected, we say $U$ is \emph{sehr sch\"on}.
\end{definition}

Sch\"on-ness was first introduced in \cite{tevelev2007tori}, though we use the definition of \cite{hacking}, which is equivalent by \cite[Lemma 2.7]{hacking}. Sehr sch\"on-ness is introduced in the forthcoming PhD thesis \cite{joris}. We use the following result from \cite{joris}, which generalizes \cite[Theorem 2.5]{hacking} to fans with weights and rephrases it in the language of positive geometry.

\begin{proposition}[{\cite[Proposition 2.2.12]{joris}}]\label{prop:joris}
Let $X$
be a projective variety of dimension $n$, and $Y$ a codimension one subvariety of $X$ such that $X \setminus Y \subset \mathbb{C}^N$ is a smooth and sehr schön very affine variety. Then the combinatorial rank of $(X,Y)$ is 
\[
\operatorname{cr}(X,Y)
=
(-1)^{n-\ell}
\sum_{k\ge0} (-1)^k f_k,
\]
where $L$ is the $\ell$-dimensional lineality space of
$\mathrm{Trop}(X \setminus Y)$, and the $f_k$ are the entries of the
weighted $f$-vector of the quotient fan
$\mathrm{Trop}(X \setminus Y)/L.$
\end{proposition}

The log canonical compactification of $\mathcal{Y}(3,6)$ was described in \cite{HKTdelPezzo}. It is identified with Naruki's cross-ratio variety \cite{Naruki}. This embedding was shown to be given by the canonical forms of the real regions in \cite{cubicsurfaces}. The following proposition gives an independent proof that the real combinatorial rank equals the combinatorial rank for $\mathcal{Y}(3,6)$ and $\mathcal{X}(3,6)$. Computations used in the proof are available at \cite{zenodo}.

\begin{proposition}
    The combinatorial ranks of $\mathcal{X}(3,6)$ and $\mathcal{Y}(3,6)$ are $126$ and $150,$ respectively. These equal the corresponding real combinatorial ranks. 
\end{proposition}

\begin{proof}
    The real regions of $\mathcal{X}(3,6)$ are positive geometries and their canonical forms are identified in \cite[Sec. 3.3.1]{nick2024x3n}.
    A direct computation shows that these forms span a linear space of dimension $126$. This is the real combinatorial rank. The open configuration spaces $\mathcal{X}(3,6)$ and $\mathcal{X}(3,7)$ are sehr sch\"on; Corey's analysis of the initial degenerations shows that the strata are smooth and irreducible \cite[Theorem 6.1, Lemma 7.1]{corey2021initial}. Thus by Proposition \ref{prop:joris}, the combinatorial rank is given by the dimension of the homology of the tropical Grassmannian $\mathrm{Trop}(\mathrm{Gr}(3,6))$. This number was computed to be $126$ by \cite[Theorem 5.4]{tropgr}. 
    
    The real regions of $\mathcal{Y}(3,6)$ are also positive geometries \cite{posdelPezzo}.
    Their canonical forms span a space of dimension $150$ and realize the log canonical embedding, as shown in \cite[Theorem 10.5 and Section 11]{cubicsurfaces}.  We note that the computation in \cite[Section 11]{cubicsurfaces} is due to Artyom Lisitsyn. He also independently obtained the number $126$ appearing in this proposition and the number $7470$ for $\mathcal{X}(3,7)$, which we discuss below.
\end{proof}

In addition, we obtain all quadratic equations that vanish on the log canonical compactification of $\mathcal{Y}(3,6)$ and on the closure of the variety parametrized by the canonical forms of the real regions of $\mathcal{X}(3,6)$. We denote the latter (closed) variety by $\mathcal{X}_{\mathrm{cf}}(3,6)$.

\begin{proposition}
The homogeneous vanishing ideals of $\mathcal{X}_{\mathrm{cf}}(3,6)$ and of the log canonical compactification of $\mathcal{Y}(3,6)$ have $6616$ and $9605$ minimal quadratic generators respectively.  
\end{proposition}

As far as we are aware, this is the first explicit description of the quadratics which vanish on these two varieties in the literature. As in previous examples, the generators we computed can be found in our supplementary materials \cite{zenodo}. We also checked that the dimension of each of these ideals is four. This immediately implies that the full ideal of each embedding is a minimal prime of its respective ideal of quadratics and suggests the following conjecture. Verifying this conjecture is unfortunately out of the reach of current computational methods.   
\begin{conjecture}
The vanishing ideals of $\mathcal{X}_{\mathrm{cf}}(3,6)$ and of the log canonical compactification of $\mathcal{Y}(3,6)\subset \mathbb{P}^{149}$ are quadratically generated. 
\end{conjecture}

In addition, we checked that $\mathcal{X}_{\mathrm{cf}}(3,6)$ is birational to  $\mathcal{X}(3,6)$. By Theorem \ref{thm:mainnew}, to show that $\mathcal{X}_{\mathrm{cf}}(3,6)$ is the log canonical compactification of $\mathcal{X}(3,6)$, it would suffice to show that the log canonical bundle is very ample. To our knowledge, this is an open question.

We now turn to the configuration space $\mathcal{X}(3,7)$. Its real regions are positive geometries and their canonical forms were also computed in \cite[Sec. 3.3.2]{nick2024x3n}. To our knowledge, it is currently not known if the canonical forms of the real regions of the configuration space $\mathcal{X}(3,7)$ realize its log canonical embedding. The discussion above suggests this is the case if the log canonical bundle of $\mathcal{X}(3,7)$ is very ample. However, we checked numerically that the canonical forms still define at least a birational map to $\mathbb{P}^{7469}$. The dimension of the linear space spanned by the canonical forms is equal to $7470$ and again equals the dimension of the homology of the tropical Grassmannian $\mathrm{Trop}(\mathrm{Gr}(3,7))$, computed in \cite[Theorem 2.1]{drawingplanes2009}. The moduli space of marked del Pezzo surfaces $\mathcal{Y}(3,7)$ is even less studied: it is currently unknown if all of its regions are positive geometries and what their canonical forms are. Thus, studying the canonical forms and equations for log canonical embeddings of $\mathcal{X}(3,7)$ and $\mathcal{Y}(3,7)$ are interesting questions for future research.

\subsection{$\cm_{0,n}$ and Parke--Taylor Varieties}
In this section we connect the \emph{Parke--Taylor variety} introduced in \cite{parketaylorvar} with our embedding via canonical forms of the configuration spaces described in the previous section. We begin with a brief overview of the Parke--Taylor variety.

Let $\sigma \in S_n$ be a permutation of $[n]$. The Parke--Taylor function associated with $\sigma$ is the rational function
\[
f_\sigma = \frac{1}{\prod_{i = 1}^n p_{\sigma_i \sigma_{i+1}}},
\]
where $p_{ij}$ are the Pl\"ucker coordinates on (the affine cone over) the Grassmannian $\gr(2,n)$ and the indices of $\sigma$ are taken modulo $n$. These functions appear in the Parke--Taylor formula for maximally helicity violating amplitudes \cite{Parke:1986gb} in quantum field theory. In \cite{parketaylorvar} the first two authors of this paper introduced the Parke--Taylor variety $\pt_n$ as the closure of the map
\begin{align*}
\varphi_n: \gr^\circ(2,n) &\to\pp^{n!-1}, \\
            p &\mapsto (f_\sigma(p) : \sigma \in S_n).
\end{align*}
The \emph{open Parke--Taylor variety} is $\pt_n^\circ = \image(\varphi_n)$ and $\pt_n$ is its Zariski closure. Two main results of \cite{parketaylorvar} relate $\pt_n$ to the moduli space $\mathcal{M}_{0,n}$ of $n$ distinct points in $\pp^1$ up to projective automorphisms. Namely, they show that $\pt_n^\circ \cong \cm_{0,n}$ and, more surprisingly, that $\pt_n$ is Keel and Tevelev's log canonical embedding of the Deligne--Knudsen--Mumford compactification $\overline{\cm}_{0,n}$ from \cite{KeelTevelevM0n}. The following result shows that these results of \cite{parketaylorvar} are simply an instance of Theorem \ref{thm:mainnew}. 

\begin{theorem}
The Parke--Taylor variety $\pt_n$ is the log canonical compactification of $\cm_{0, n}$. 
\end{theorem}
\begin{proof}
Let $D$ denote the boundary of the Deligne--Knudsen--Mumford compactification $ \overline{\cm}_{0,n}$. The Parke--Taylor forms are the canonical forms of the regions of $(\overline{\cm}_{0,n}, D)$ \cite[Sec. 5]{wondertopes}.
The log canonical ring is Koszul and thus generated in degree one by \cite[Theorem 1.2]{KeelTevelevM0n}, and the image is projectively normal, so the log canonical bundle is very ample \cite[Exercise 5.14(d)]{Hartshorne}. We have $g(\overline{\cm}_{0,n}, D)=0$ by \cite[Sec. 6.8]{brown-dupont}. Moreover, $(\overline{\cm}_{0,n}, D)$ is a modification of a real hyperplane arrangement, thus $\mathrm{cr}_\mathbb{R}(\overline{\cm}_{0,n},D)=\mathrm{cr}(\overline{\cm}_{0,n},D)$. Thus Theorem \ref{thm:mainnew} applies.
\end{proof}

One may ask whether this result extends to the configuration spaces $\mathcal{X}(3,n)$ from the previous subsection. Consider the \emph{generalized Parke--Taylor form}
\[
f_\sigma^{(k)} = \frac{1}{\prod_{i = 1}^n p_{\sigma_i \sigma_{i+1} \ldots \sigma_{i+k-1}}}
\]
which is well-defined on $\gr^\circ(k,n)$ and thus on $\mathcal{X}(k,n)$. 
For $k=3$ and $n=6$ this form appears in the list of canonical forms used in our embedding of $\mathcal{X}(3,6)$. These functions appear naturally in the theory of \emph{CEGM amplitudes} \cite{nick2024x3n}. However, as we saw in the previous subsection, these forms do not suffice to obtain the full log canonical embedding. In particular, the other three types of forms utilized in the embedding of $\mathcal{X}(3,6)$ are not in the span of the generalized Parke--Taylor forms. Indeed, in this case the real combinatorial rank is $126$ but there are only $60$ linearly independent Parke--Taylor forms. 
Despite this, it is natural to ask if the Parke--Taylor forms still yield a birational model for  $\mathcal{X}(3,n)$. Unfortunately, this is already false for $k = 3$ and $n = 6$. The map $\varphi_{3, 6}$ given by
\begin{align*}
\varphi_{3, 6}: \gr^\circ(3,6) &\to \pp^{6!-1}, \\
            p &\mapsto (f_\sigma^{(3)} : \sigma \in S_6)
\end{align*}
is not injective when restricted to $\mathcal{X}(3, 6)$. A direct computation with many points reveals that the generic fiber in this case has cardinality $2$. An example of two points whose generalized Parke--Taylor forms all agree can be found in \cite[Example 6.4]{segre}. 

\medskip
\noindent\textbf{AI Use.} 
We used GPT-5.6 Sol and Claude Opus 5 to assist in finding and parsing prior work. We used Claude Opus 5 to generate a mock review which caught many typos and notational inconsistencies, and fixed a gap in the proof of Lemma \ref{lem:canformslogarithmic}. We also used Claude Sonnet 5 to assist in writing \texttt{HyperplaneArrangements.jl}. We take full responsibility for the content of the paper and supplementary materials.

\medskip
\noindent\textbf{Acknowledgments.} 
The authors would like to thank Olivia Dumitrescu, Joris Koefler, Arne Kuhrs, Artyom Lisitsyn, Rainer Sinn, and Bernd Sturmfels for helpful discussions. They also thank Jenia Tevelev for pointers to the relevant literature. 

\smallskip

This research is funded by the European Union (ERC, UNIVERSE PLUS, 101118787). Views and opinions expressed are, however, those of the authors only and do not necessarily reflect those of the European Union or the European Research Council Executive Agency. Neither the European Union nor the granting authority can be held responsible for them. The third author was supported during this research by NSF graduate research fellowship no. 2023358166 and the Chancellor's Postdoctoral Fellowship at University of California, Davis.

\bibliographystyle{plain}
\bibliography{literature}
Institute\end{document}